\documentclass[bj,authoryear]{imsart}
\usepackage{amsmath}
\usepackage{fancyhdr}
\usepackage{mathrsfs,epsfig,morefloats}
\RequirePackage[OT1]{fontenc}
\RequirePackage{amsthm,amsmath,natbib}
\RequirePackage[colorlinks,citecolor=blue,urlcolor=blue]{hyperref}
\usepackage{epstopdf}
\usepackage{bm}
\usepackage{amsfonts}
\usepackage{mathrsfs,epsfig,morefloats}
\usepackage{amsmath}
\usepackage{amsthm}
\usepackage{amssymb}
\usepackage{color}
\usepackage{lscape}
\usepackage{rotating}
\usepackage{subfigure}
\usepackage{graphicx}

\usepackage{enumerate}

\usepackage{amssymb}
\usepackage{multirow}
\usepackage[scientific-notation=true]{siunitx}
\usepackage{tabularx,ragged2e,booktabs,caption}

\newcommand{\X}{{\bm X}}
\newcommand{\x}{{\bm x}}
\newcommand{\Y}{{\bm Y}}
\newcommand{\y}{{\bm y}}

\newcommand{\ba}{{\bm a}}
\newcommand{\bu}{{\bm u}}
\newcommand{\V}{{\bm V}}
\newcommand{\bv}{{\bm v}}
\newcommand{\W}{{\bm W}}
\newcommand{\T}{{\bm T}}
\newcommand{\bS}{{\bm S}}
\newcommand{\B}{{\bm B}}

\newcommand{\w}{{\bm w}}
\newcommand{\e}{{\bm e}}

\newcommand{\A}{{\bm A}}

\newcommand{\E}{{\mathbb E}}

\newcommand{\R}{{\bm R}}
\newcommand{\PP}{{\bm P}}
\newcommand{\Q}{{\bm Q}}
\newcommand{\D}{{\bm D}}

\renewcommand{\P}{\mathbb{P}}
\newcommand{\coc}{\color{cyan}}

\usepackage{multirow}
\usepackage{tabularx,booktabs,caption}
\newcolumntype{C}[1]{>{\Centering}m{#1}}

\newtheorem{theorem}{Theorem}[section]
\newtheorem{lemma}[theorem]{Lemma}
\newtheorem{corollary}[theorem]{Corollary}
\numberwithin{equation}{section}
\newtheorem{remark}[theorem]{Remark}

\newtheorem{assumption}{Assumption}[section]

\endlocaldefs

\begin{document}

\begin{frontmatter}
\title{\large\bf Sampling without replacement from a high-dimensional finite population}
\runtitle{Sampling without replacement}

\begin{aug}
\author[A]{\fnms{Jiang} \snm{Hu}\ead[label=e1]{huj156@nenu.edu.cn}}
\author[B]{\fnms{Shaochen} \snm{Wang}\ead[label=e2]{mascwang@scut.edu.cn}}
\author[C]{\fnms{Yangchun}
\snm{Zhang}\ead[label=e3]{zycstatis@shu.edu.cn}\thanks{Corresponding author}}
\and
\author[D]{\fnms{Wang} \snm{Zhou}%
\ead[label=e4]{stazw@nus.edu.sg}%
\ead[label=u1,url]{}}

\address[A]{KLASMOE and School of Mathematics $\&$ Statistics, Northeast Normal University,
Changchun, 130024, P.R. China. \printead{e1}}

\address[B]{School of Mathematics,
South China University of Technology, Guangzhou, 510640, P.R. China.\printead{e2}}

\address[C]{Department of Mathematics, Shanghai University, Shanghai, 200444, P.R. China.\printead{e3}}

\address[D]{Department of Statistics and Data Science, National University of Singapore, 117546, Singapore.
\printead{e4}}

\runauthor{Hu et al.}

\affiliation{Some University and Another University}

\end{aug}

\begin{abstract}
  It is well known that  most of the existing theoretical results in statistics are based on the assumption that the sample is generated
  with replacement  from an infinite population.
  However, in practice, available samples are almost always collected without replacement.  If the population is  a finite set of real numbers, whether we can still safely use the results from  samples drawn without replacement becomes an important  problem. In this paper, we focus on the eigenvalues of  high-dimensional sample covariance matrices generated without replacement from finite populations. Specifically,
  we derive the Tracy-Widom laws for their largest eigenvalues and apply these results to parallel analysis. We provide new insight into the permutation methods proposed by Buja and Eyuboglu in [Multivar Behav Res. 27(4) (1992) 509--540]. 
  Simulation and real data studies  are conducted to demonstrate our results.
\end{abstract}


\end{frontmatter}
\section{Introduction}\label{section1}
Let
$
        {\mathbf U}_{p \times N}=(\bu^{\top}_1,\ldots,\bu^{\top}_p)^{\top}
$ be a finite population  matrix of  dimension $p$ and size $N$,
    where $\bu_i=(u_{i1},\ldots,u_{iN})$, $1\leq i\leq p$,  is the i-th row of ${\mathbf U}_{p \times N}$.
    Suppose that for each $i=1,\ldots,p$, ${\w}_i=(w_{i1},\ldots,w_{in})$  is generated by sampling without replacement from the $N$ coordinates of $\bu_i$ with equal probability.
    Then, the sample covariance matrix from the finite population is given by
    \begin{eqnarray}\label{al2}
    {\bS}_{n}=\frac{1}{n}\W_n\W_n^{\top},
    \end{eqnarray}
    where $\W_n=({\w}_1^{\top},\ldots,{\w}_p^{\top})^{\top}$. 
In this paper, our aim is to derive the asymptotic properties  of the  eigenvalues of ${\bS}_{n}$ in the large dimensional framework, that is, cases where the dimension and sample size proportionally tend to infinity.

    Sample covariance matrices have been extensively investigated since proposed by \cite{Wishart28G}. For fixed sample dimension cases,  one can see any multivariate statistical analysis (MSA) textbook (e.g., \citep{And}). However, when the dimension of  the sample  is large or even larger than the sample size, these classical theories  may no longer be applicable. 
   Since Marcenko and Pastur discovered the well-known global spectral distribution (MP law) in their seminal work \cite{MP}, there has been much literature devoted to the spectral property of the large-dimensional sample covariance matrix and its variants.  For the global properties of eigenvalues of sample covariance matrices, see \cite{Baibook}, where one can also find many references related to this topic.  For the individual eigenvalue, the most interesting
   one is the largest eigenvalue, since it accounts for as much of the variability in data
   as possible. The limiting distribution of  the largest eigenvalue of the large-dimensional sample covariance matrix was first studied and shown to follow the Tracy-Widom (TW) distribution by \cite{Johan} (complex Gaussian populations)
     and  \cite{Johnstone} (real Gaussian populations)
    under the assumption that all components of the population are independent and 
    identically distributed 
(i.i.d.). 
For the follow-up work, see \cite{Karoui,Ona,Baopanzhou,Lee}.
    
The coordinates of the population are linear combinations of independent random variables, which is a common condition in the above results. To relax this condition, many researchers have considered more general cases.
 For example, \cite{Bai08} considered the global spectral distribution of the sample covariance matrix without an independence structure in columns, \cite{Baopanzhou2} and \cite{PAY}  considered the TW law of the largest eigenvalue of the sample correlation matrix, \cite{GHPY} considered the linear spectral statistics of the sample correlation matrix, \cite{Bao} considered the TW law of the largest eigenvalue of the Kendall rank correlation matrix, \cite{Hu} and \cite{Wen} considered the linear spectral statistics and the largest eigenvalue of the sample covariance matrix, respectively, in elliptical distributions, and \cite{TZZ} considered the largest eigenvalue of the sample covariance matrix generated by VARMA.

    It is worth noting  that the sample covariance matrix from finite population ${\bS}_{n}$, which is a classical statistic in MSA,  is  a  kind of  sample covariance matrix without an independence structure. However, it does  not belong to any model presented in the above references. This is because (i) the entries of each row of $\W_n$  are correlated and (ii) the populations of different rows of $\W_n$ could be different.   Here we introduce Spearman's rank correlation  matrix  and the permutation matrix in parallel analysis (PA)  as  two examples  of our matrix model.

 \begin{itemize}
 	\item[(I):]
Spearman's rank correlation coefficient  is a classical nonparametric measure of  correlation  between two random variables. Suppose that we have $n$ i.i.d. samples $\Y_{1}, \ldots, \Y_{n}$, where $\Y_{k}$ is a $p$-vector consisting of i.i.d. random variables $Y_{1 k}, \ldots, Y_{p k}, k=1, \ldots, n .$ Spearman's rank correlation matrix
is defined by
\begin{align}\label{spear}
	\mathfrak{R}_{p}=\left(\mathfrak{r}_{k l}\right)_{1 \leq k, l \leq p},
	\end{align}
where $\mathfrak{r}_{k l}$ is Spearman's rank correlation coefficient between the $k$-th and $l$-th rows of $\left[\Y_{1}, \ldots, \Y_{n}\right]$.
Actually, $\mathfrak{R}_{p}$ can be expressed in the form of \eqref{al2} with $u_{ij}=\frac{\sqrt{12 }}{\sqrt{n^{2}-1}}(j-\frac{n+1}{2})$, $i=1,\dots,p$, $j=1,\dots,N$ and $n=N$.  See \cite{Bai08,BaoL15,Baospe} for more details.

%

\item[(II):]
PA  is a popular method that is used to select the number
    of factors in the factor model. PA is an applicable method because the largest singular value of the permuted data matrix decreases compared with that of the original value if there is indeed at least one factor (see \cite{DO,DE} for further explanation).
    Suppose the sample   matrix $\B=(b_{ij})_{p\times n}$ satisfies the signal-plus-noise model, i.e., $\B=\bS+\bm\Psi$, where $\bS$ is a $p\times n$ signal matrix of
    low rank and $\bm\Psi$ is a $p\times n$ noise matrix.
    In this case, the permutation matrix in \cite{DO,DE} turns into $\W$ of \eqref{al2} with $N = n$.
   \end{itemize}

We highlight  two main contributions of the present paper.
First, we theoretically investigate some spectral properties of  high-dimensional covariance matrices from finite populations, which have not attracted much attention in modern statistics. Second,
we practically characterize the asymptotic distribution of the largest eigenvalue of the high-dimensional standard permutation matrix in PA, which provides  a direct procedure without permutation.

The remainder of this paper is organized as follows. In Section \ref{mainresults}, we introduce the necessary notations and state the  main results, which include the global law, the edge universality and the fluctuation of the largest eigenvalue of the  sample covariance matrix ${\bS}_{n}$.
In Section \ref{section2}, we introduce
some applications of our main results in high-dimensional statistical inference to parallel analysis.
Related simulations and real data analyses are also conducted to demonstrate  our results.
Proofs of the main theorems and some preparatory lemmas are given  in Section \ref{TP}.   

\section{Notations and statements of main results}\label{mainresults}
\subsection{Main results}
To easily achieve our goal, we write ${\mathbf U}_{p \times N}$ as
\begin{eqnarray*}
{\mathbf U}_{p \times N}=\T^{\frac{1}{2}}{\mathbf V}_{p \times N},
\end{eqnarray*}
where $\T=\text {{\rm diag}}\{t_1,\ldots,t_p\}$ is a positive definite $p$-dimensional diagonal matrix and {${\mathbf V}_{p \times N}=(v_{jk})_{1\leq j\leq p,~1\leq k \leq N}$} is a $p\times N$ matrix. $\T$ is chosen such that the rows of ${\mathbf V}_{p \times N}$ are standardized. The rows of ${\mathbf U}_{p \times N}$ and ${\mathbf V}_{p \times N}$ are denoted as $\bu_1,\ldots, \bu_p$ and $\bv_1,\ldots, \bv_p$, respectively.
After scaling the observations ${\y}_i={\w}_i/\sqrt{n},~{\x}_i={\y}_i/\sqrt{t_i},~i=1,\ldots,p$ and setting $\X_n=({\x}_1^{\top},\ldots,{\x}_p^{\top})^{\top}=(x_{jk})_{1\leq j\leq p,~1\leq k \leq n}$, $\Y_n=({\y}_1^{\top},\ldots,{\y}_p^{\top})^{\top}=(y_{jk})_{1\leq j\leq p,~1\leq k \leq n}$, we can rewrite \eqref{al2} as
\begin{align}\label{syx} {\bS}_{n}=\Y_n\Y_n^{\top}&=\T^{\frac{1}{2}}\X_n\X_n^{\top}\T^{\frac{1}{2}}.
\end{align}
Henceforth, we denote $\lambda_p(\A)\leq \cdots \leq \lambda_1(\A)$ as the ordered eigenvalues of
a $p\times p$ Hermitian matrix $\A$. The empirical spectral distribution (ESD) of $\T$ is
\begin{eqnarray}\label{ESD}
H_n(\lambda):=\frac{1}{p}\sum_{j=1}^p {\mathbb I}_{\{\lambda_{j}(\T) \leq \lambda\}},\quad \lambda\in\mathbb{R}
\end{eqnarray}
and that of ${\bS}_{n}$ is
\begin{eqnarray*}
    F_n(\lambda):=\frac{1}{p}\sum_{j=1}^p {\mathbb I}_{\{\lambda_{j}({\bS}_{n}) \leq \lambda\}},\quad \lambda\in\mathbb{R}.
\end{eqnarray*}
Here and throughout the paper, ${\mathbb I}_{A}$ represents the indicator function of event $A$. The Stieltjes transform of $F_n$ is given by
\begin{eqnarray*}
     m_n(z)=\int \frac{1}{x-z}{\rm d}F_n(x),
\end{eqnarray*}
where $z=E+i\eta \in \mathbb{C}^{+}$.
    
In addition, we need a crucial parameter $\xi_+$ that satisfies
    \begin{eqnarray}\label{intro2}
      \int \left(\frac{\lambda\xi_+}{1-\lambda\xi_+}\right)^2{\rm d}H_n(\lambda)=p/n.
    \end{eqnarray}
    It is easy to check that the solution to \eqref{intro2} in $[0,1/\lambda_1(\T))$ is unique. 
    See \cite{Karoui} for more discussion of $\xi_+$.
Now, we state our main assumptions as follows.

    \begin{assumption}\label{ass1}
    \begin{itemize}
      \item[(i)] (On dimensionality) We assume
      \begin{eqnarray}\label{intro1}
    {c_n:=p/n\rightarrow c_0\in(0,\infty),~~y_n:=n/N\rightarrow y\in[0,1]}
    \end{eqnarray}
    as $n\to\infty$. For simplicity, we assume that there are
      two  positive  constants $c_1$, $C_1$ such that $c_1 <c_n <C_1$ for all $n\geq1$.
      \item[(ii)] (On $\T$) We assume that $\T$ is  diagonal with ${\liminf}_n \lambda_p(\T)>0$, ${\limsup}_n \lambda_1(\T)<\infty$,
      
      \begin{eqnarray}\label{laxi}
       \limsup_n \lambda_1(\T)\xi_+<1,
      \end{eqnarray}
        and  there exists a deterministic distribution $H$ such that $H_n \to H$ as $n\to \infty$.
        \item[(iii)]  (On $\V$)  We assume that the elements in ${\mathbf V}_{p \times N}$  satisfy   for each $j=1,\ldots,p$, $\sum_{k=1}^N v_{jk}=0$, $N^{-1}\sum_{k=1}^N v_{jk}^2=1$. 
        In addition, for any positive integer $ m$, there is a constant
      $C_m$ such that
      \begin{align}\label{mmoments}
       \frac{1}{N}\sum_{k=1}^N|v_{jk}|^m\leq C_m, \ j=1,\ldots,p.
      \end{align}
      \end{itemize}
    \end{assumption}
    \begin{remark}\label{rema12}
    We remark that the aim of Assumption \eqref{mmoments} is to make  moments of $\sqrt{n}x_{ij}$ bounded, i.e., $\mathbb E|\sqrt{n}x_{ij}|^m \leq C_m$, which justifies the large deviation of the strong MP law and the moments' comparison procedure. Details can be found in Section \ref{TP}.
    \end{remark}

    Our main results on the sample covariance matrix ${\bS}_{n}$ from a finite population are formulated as follows.
    First, we present the global law of ${\bS}_{n}$.

    \begin{theorem}\label{th1MP}
    Suppose that ${\bS}_{n}$ satisfies Assumption \ref{ass1}. 
    Then, as $n\to \infty$, $F_n$ almost surely converges  to a probability distribution $F_{c_0,H}$,  whose Stieltjes transform $m_{c_0,H}=m_{c_0,H}(z)$ is determined by
    \begin{eqnarray}\label{th11}
    m_{c_0,H}=\int \frac{1}{t(1-c_0-c_0zm_{c_0,H})-z}{\rm d}H(t).
    \end{eqnarray}
    \end{theorem}
    
    We further denote by $\Q:=n^{-1}\T^{\frac{1}{2}}\mathcal{X}\mathcal{X}^{\top}\T^{\frac{1}{2}}$, where $\mathcal{X}$ is a $p\times n$ data matrix with i.i.d. $N(0, 1)$ variables. Our main theorem on the edge universality of ${\bS}_{n}$  can be formulated as follows, and  its proof is deferred to Section \ref{TP}. Let
    \begin{eqnarray}\label{E+}
    	E_+=\frac{1}{\xi_+}\left(1+c_n^{-1}\int\frac{\lambda\xi_+}{1-\lambda\xi_+}{\rm d}H_n(\lambda)\right).
    \end{eqnarray}
    \begin{theorem}\label{comparison}
    Suppose that {Assumption \ref{ass1} } holds, and $N=n$.
     Then, there exist positive constants $\varepsilon$ and $\delta$ such that for any $s\in \mathbb{R}${\color{red},} 
    \begin{align*}
    &\mathbb P\left(n^{2/3}(\lambda_1(\Q)-E_+)\leq s-n^{-\varepsilon}\right)-n^{-\delta}\\
    \leq& \mathbb P\left(n^{2/3}(\lambda_1({\bS}_{n})-E_+)\leq s\right)\\
    \leq& \mathbb P\left(n^{2/3}(\lambda_1(\Q)-E_+)\leq s+n^{-\varepsilon}\right)+n^{-\delta},
    \end{align*}
    holds for sufficiently large $n$.
    \end{theorem}

    One typical application of Theorems \ref{th1MP} and \ref{comparison} lies in Spearman's correlation matrix \eqref{spear}. In this case, $\T$ is identity, and it is straightforward to verify that Assumption \ref{ass1}  holds. Thus, we have the following corollary, which was first derived in \cite{Baospe}.
  	\begin{corollary}\label{cor2.4}
  		Suppose that $c_n:=p/n\rightarrow c_0\in(0,\infty)$. Then, the ESD of Spearman's rank correlation matrix
  		$\mathfrak{R}_{p}$  almost surely  converges  to  the standard MP law, and
  		\begin{eqnarray*}
  			n^{\frac{2}{3}}\left(\frac{\lambda_1(\mathfrak{R}_{p})-\underline{E}_+}{\underline{\gamma}_0}\right)\Rightarrow TW_1,
  		\end{eqnarray*}
  	where $\underline{E}_+=(1+\sqrt{c_n})^2$ and $\underline{\gamma}_0=c_n^{-\frac{1}{6}}\underline{E}_+^{\frac{2}{3}}$. Here and 
in the sequel
,  $TW_1$ is  the type-1 TW distribution,  $``\Rightarrow"$ means convergence in distribution.
  	\end{corollary}

%
%

    However, one cannot apply Theorem  \ref{comparison} directly in PA because in PA, the data matrix is random, and consequently, the condition $\sum_{k=1}^N v_{jk}=0$ is not satisfied. To overcome this problem, we consider modified versions of ${\bS}_{n}$, which fortunately work for all cases of $n\leq  N$.
    
    Now, we consider the general case of $n\leq  N$. We first introduce the spatial sign matrix. Let  $\bar{\Y}={\rm diag}\{\frac{1}{n}\sum_{l=1}^ny_{1l},\dots,\frac{1}{n}\sum_{l=1}^ny_{pl}\}\e_p\e_p^{\top}$,  $\bS_P=(\Y_n-\bar{\Y})(\Y_n-\bar{\Y})^{\top}$ and 
    $\D_2^{(n)}={\rm diag}\{(\bS_P)_{11},\dots,(\bS_P)_{pp}\}$. Here $\e_p$ represents the $p$-dimensional vector with elements all 1.  Then, the spatial sign matrix is defined as
    \begin{align}\label{Rn}
      \PP_n=(\D_2^{(n)})^{-\frac{1}{2}}\bS_P(\D_2^{(n)})^{-\frac{1}{2}}.
    \end{align}
    In addition, we provide an alternative assumption for  (iii) of {Assumption \ref{ass1}}; namely,
    \begin{itemize}
      \item[(iii')] {Assume that for any positive integer $m$, there is a constant 
      $C_m$ such that}
      \begin{align}\label{mmmm}
        \frac{1}{N}\sum_{k=1}^N|v_{jk}|^m\leq C_m, ~~~~j=1,\dots,p.
        \end{align}
Furthermore, we assume
      \begin{align}\label{mmmm1}
        \frac{1}{N}\sum_{j=1}^N\Big(v_{ij}-\frac{1}{N}\sum_{k=1}^Nv_{ik}\Big)^2\geq c>0,
      \end{align}
      for some constant $c$.
    \end{itemize}

    \begin{remark}
      Note that $\sum_{k=1}^N v_{jk}=0$ and $N^{-1}\sum_{k=1}^N v_{jk}^2=1$ in (iii) are no longer needed because of the standardization
      in the spatial-sign matrix.  The aim of the new condition \eqref{mmmm1} is   to bound the moments of $\sqrt{n}\frac{y_{ij}-\frac{1}{n}\sum_{k=1}^ny_{ik}}{\sqrt{(\bS_P)_{ii}}}$, i.e.,
      \begin{align*}
        \mathbb E\Big|\sqrt{n}\frac{y_{ij}-\frac{1}{n}\sum_{k=1}^ny_{ik}}{\sqrt{(\bS_P)_{ii}}}\Big|^m \leq C_m,
      \end{align*}
    which will be used   in the proof of Theorem \ref{N>n}.
    \end{remark}

    \begin{theorem}\label{N>n}

      In addition to (i) and (ii) of Assumption \ref{ass1}, we assume (iii') holds. We write $\Q_1:=\frac{1}{n}\mathcal{X}\mathcal{X}^{\top}$; then, there exist positive constants $\varepsilon$ and $\delta$ such that for any $s\in \mathbb{R}$,
      \begin{align*}
        &\mathbb P\left(n^{2/3}(\lambda_1(\Q_ 1)-\underline{E}_+)\leq s-n^{-\varepsilon}\right)-n^{-\delta}\\
        \leq &\mathbb P\left(n^{2/3}(\lambda_1(\PP_n)-\underline{E}_+)\leq s\right)\\
        \leq &\mathbb P\left(n^{2/3}(\lambda_1(\Q_1)-\underline{E}_+)\leq s+n^{-\varepsilon}\right)+n^{-\delta}
      \end{align*}
       holds for sufficiently large $n$. Here, $\underline{E}_+$ is defined in Corollary \ref{cor2.4}.
    \end{theorem}

 \begin{remark}\label{remk1}
    The above results can be generalized to the joint distribution of the largest several eigenvalues. More specifically, under the assumptions in Theorem \ref{comparison}, there exist positive constants $\varepsilon$ and $\delta$ such that for any $s_1,\ldots,s_k\in \mathbb{R}$,
    \begin{align*}
    &\mathbb P\left(n^{2/3}(\lambda_1(\Q)-E_+)\leq s_1-n^{-\varepsilon},\ldots,n^{2/3}(\lambda_k(\Q)-E_+)\leq s_k-n^{-\varepsilon}\right)-n^{-\delta}\\
    \leq&\mathbb P\left(n^{2/3}(\lambda_1({\bS}_{n})-E_+)\leq s_1,\ldots,n^{2/3}(\lambda_k({\bS}_{n})-E_+)\leq s_k\right)\\
    \leq&\mathbb P\left(n^{2/3}(\lambda_1(\Q)-E_+)\leq s_1+n^{-\varepsilon},\ldots,n^{2/3}(\lambda_1(\Q)-E_+)\leq s_k+n^{-\varepsilon}\right)+n^{-\delta}
    \end{align*}
    holds for sufficiently large $n$. Correspondingly, 
    $\lambda_1(\PP_n),\ldots,\lambda_k(\PP_n)$ have analogous properties under the conditions of Theorem \ref{N>n}.
    \end{remark}

    From Theorems \ref{comparison} and  \ref{N>n}, we have the following corollary on the largest eigenvalues of ${\bS}_{n}, \PP_n$.
    \begin{corollary}\label{TWS}
    Suppose that {Assumption \ref{ass1} } holds. Denoting
    \begin{eqnarray}\label{GAMMA0}
    \gamma_0^3=\frac{1}{\xi_+^3}\left(1+c_n^{-1}\int\left(\frac{\lambda\xi_+}{1-\lambda\xi_+}\right)^3{\rm d}H_n(\lambda)\right),
    \end{eqnarray}
    we have
    \begin{eqnarray*}
    n^{\frac{2}{3}}\left(\frac{\lambda_1({\bS}_{n})-E_+}{\gamma_0}\right)\Rightarrow TW_1.
    \end{eqnarray*}
   Additionally, under the same assumptions  in Theorem \ref{N>n}, we have
    \begin{eqnarray*}
      n^{\frac{2}{3}}\left(\frac{\lambda_1(\PP_n)-\underline{E}_+}{\underline{\gamma}_0}\right)\Rightarrow TW_1.
    \end{eqnarray*}
    \end{corollary}

The proof of Theorem \ref{comparison} is deferred to Section \ref{TP}. The proofs of Theorem \ref{N>n} and Corollary \ref{TWS} are deferred to Section A.3 of Supplementary Material (\cite{Hu23}).

\subsection{Basic notions}\label{notionnn}
What truly pertains to our discussion in the sequel is the nonasymptotic version of $F_{c_0,H}$ (i.e., $F_{c_n,H_n}$),  which can be obtained by replacing $c_0$ and $H$ with $c_n$ and $H_n$ in $F_{c_0,H}$, respectively.  More precisely, $F_{c_n,H_n}$ is the corresponding distribution function of the Stieltjes transform $m_{c_n,H_n}(z) := m_{c_n,H_n}\in \mathbb{C}^+$ satisfying the following self-consistent equation:
\begin{eqnarray}\label{local law2}
m_{c_n,H_n}(z)=\int\frac{1}{t(1-c_n-c_nzm_{c_n,H_n}(z))-z}{\rm d}H_n(t).
\end{eqnarray}
Define an $n\times n$ matrix
\begin{eqnarray}\label{fs}
    \mathcal{S}_n=\Y_n^{\top}\Y_n,
\end{eqnarray}
which shares the same nonzero eigenvalues with $\bS_{n}$. Denoting the ESD of $\mathcal{S}_n$ as $\underline{F}_n$, we have$$
\underline{F}_n=c_nF_n+(1-c_n)\mathbb{I}_{[0,\infty)},
$$
which converges almost surely to
\begin{align}\label{local law_c}
\underline{F}_{c_0,H}=c_0F_{c_0,H}+(1-c_0)\mathbb{I}_{[0,\infty)}.
\end{align}
Then, the n-dependent version of \eqref{local law_c} is
\begin{align}\label{local law_c1}
\underline{F}_{c_n,H_n}=c_nF_{c_n,H_n}+(1-c_n)\mathbb{I}_{[0,\infty)}.
\end{align}
We also obtain
\begin{align}\label{local law_c2}
\underline{m}_{c_n,H_n}=c_nm_{c_n,H_n}-\frac{1-c_n}{z},
\end{align}
which satisfies the following self-consistent equation:
\begin{eqnarray}\label{local law_z}
z=-\frac{1}{\underline{m}_{c_n,H_n}}+c_n\int \frac{t}{1+t\underline{m}_{c_n,H_n}}\text{d}H_n(t).
\end{eqnarray}
For simplicity, we briefly use the notation
\begin{align*}
&\tilde{m}(z):=m_{c_n,H_n}(z),~~\tilde F:=F_{c_n,H_n},\\
&\tilde{\underline{m}}(z):=\underline{m}_{c_n,H_n}(z),~~ \underline{\tilde{F}}:=\underline{F}_{c_n,H_n}.
\end{align*}

From the discussions in \cite{SAC}, we know that $\underline{F}_{c_n,H_n}$ has a continuous derivative $\underline{\rho}_0$ on $\mathbb R \backslash \{0\}$. The inequality \eqref{laxi} guarantees that the density $\underline{\rho}_0$ exhibits
a square-root-type behavior at the rightmost endpoint of its support; see Theorem 3.1 in \cite{Baopanzhou}.
The rightmost boundary of the support of $\underline{\rho}_0$ is $E_{+}$ defined in \eqref{E+}, that is, $E_{+}=\inf\{x\in \mathbb R: \underline{\tilde{F}}(x)=1\}$. Moreover, the parameter $\xi_{+}$ defined by \eqref{intro2} satisfies $\xi_{+}=-\lim_{\eta \to 0} \tilde{\underline{m}}(E_{+}+i\eta)=-\tilde{\underline{m}}(E_{+})$.

\section{Applications and numerical studies}\label{section2}

\subsection{High-dimensional PA} As  Spearman's rank correlation matrix has been discussed in the unpublished paper by \cite{Baospe}, in this section, we only focus on the high-dimensional PA. 
   Recall the    signal-plus-noise  matrix  $\B=\bS+\bm\Psi$, and recall that
    $\bm\Psi$ has the following structure:
    $$
    \bm\Psi=\tilde{\T}^{\frac{1}{2}}\bm\Phi,
    $$
    where $\tilde{\T}={\rm diag}\{\tilde{t}_1,\ldots,\tilde{t}_p\}$ is a diagonal matrix with $\|\tilde{\T}\|$ and $\|\tilde{\T}^{-1}\|$ bounded in $p$, $\bm\Phi$ is a $p\times n$ collection of i.i.d. noise $\phi_{ij}$ satisfying
    \begin{align}\label{asuph1}
    \E \phi_{ij}=0,~~\E \phi_{ij}^2=1.
    \end{align}
     Moreover, we assume that every $\phi_{ij}$ has a uniform subexponential decay, that is, there exists a constant $\vartheta>0$ such that for $u>1$
    \begin{align}\label{asuph2}
    \P(|\phi_{ij}|>u)\leq \vartheta^{-1}\exp(-u^\vartheta).
    \end{align}
 
    Denote  $\bm S=(s_{ij})_{p\times n}$, and  assume that     for any positive integer $ m$, there exists a constant
     $C_m$ such that
    \begin{align}\label{asus}
    \mathbb E |s_{ij}|^m\leq C_m,~~1\leq i\leq p,~1\leq j\leq n.
    \end{align}
    
    Recall the standard factor model $\B=\bm H\bm \Lambda^{\top}+\bm \Psi$, where
    $\bm H$ is the $p \times r$ ($r$ is fixed) matrix containing the factor values $h_{ij}$ and $\bm \Lambda$ is the
    $n \times r$ factor loading matrix with entries $\lambda_{jk}$.
    The first term is the signal due to the factor component,  whose rank is at most $r$. Thus, the factor model falls into
    the signal-plus-noise framework.

    We start with a $p\times n$ data matrix $\B$ with measurements
    given by $b_{ij}$, $i=1,\ldots,p,~j =1,\ldots,n$.
    We generate a matrix $\B_{\pi}$ by separately permuting the entries
    in each row of $\B$. Here $\pi= (\pi_1,\pi_2,\ldots,\pi_p)$ is a permutation
    array, which is a collection of permutations $\pi_i$ of $1,\ldots,n$. The permutation
    $\pi_i$ permutes the $i$-th row of $\B$; hence, $\B_{\pi}$ has $(\B_{\pi})_{ij} = \B_{i\pi_i(j)}$ entries.

     \cite{BAE} suggested  selecting
     the first factor if the top singular
    value of $\B$ is larger than a fixed percentile of the top singular values of the permuted matrices. Here, one can
    use the median or the $95\%$ or $100\%$ percentiles. If the first factor is selected,
    then we select the second factor when the second largest singular value of
    $\B$ is larger than the second singular value of the permuted matrices. We
     repeat the same procedure until a factor is not selected. 

   We try to calculate the quantiles theoretically. Based on the testing strategy introduced by \cite{DE}, we consider an analogous strategy to compare the top singular
   value of $\B$ and a certain percentile of the top singular values of the permuted matrices
    $(\hat{\B}_1)_{\pi}$ or $(\hat{\B}_2)_{\pi}$, where
    $$
    \hat{\B_1}:=\B-\bar{\B}~~\text{and}~~
    \hat{\B_2}:=(\D_2^{(n)})^{-1/2}(\B-\bar{\B}).
    $$
    $\bar{\B}$ and $\D_2^{(n)}$ are obtained in the same way as similar matrices in the construction of the
    spatial sign matrix $\PP_n$ in \eqref{Rn}.
    
    Let $\varrho_n(\hat{\B}_{\ell})$ be the largest eigenvalue of $\frac{1}{n}(\hat{\B}_{\ell})_{\pi}(\hat{\B}_{\ell})_{\pi}^{\top}$ for a random permutation $\pi$, $\ell=1,2$.  As the number of all the permutations of $(\hat{\B}_\ell)_{\pi}$ is $(n!)^p$, the probability mass function of $\varrho_n(\hat{\B}_\ell)$ given $
    \B$ is
    $$
    \mathbb P\left(\varrho_n(\hat{\B}_{\ell})=\{\lambda_1(\frac{1}{n}(\hat{\B}_{\ell})_{\pi}(\hat{\B}_{\ell})_{\pi}^{\top}), \pi\in \Pi(\hat{\B}_{\ell})\}_k|\B\right)=\frac{1}{(n!)^p},~k=1,\ldots,(n!)^p,~\ell=1,2{\color{red},}
    $$
    where $\Pi(\hat{\B}_{\ell})$ represents the set of all permutations of $(\hat{\B}_{\ell})_{\pi}$ and $\{A\}_k$ denotes the $k$-th element in a set $A$.
  For $\varrho_n(\hat{\B}_{\ell})$, we have the following corollary, whose proof is deferred to the Supplementary Material.

    \begin{corollary}\label{sta}
     Suppose that (i) and (ii) of Assumption \ref{ass1} hold when $\T$ is replaced by $(\D_2^{(n)})^{-1}$. Moreover, assume that \eqref{asuph1}, \eqref{asuph2} and \eqref{asus} hold.
      Then, we have that  
      \begin{align}\label{b2}
      n^{\frac{2}{3}}\left(\frac{\varrho_n(\hat{\B}_2)-\underline{E}_+}{\underline{\gamma}_0}
      \right)\Rightarrow TW_1,
     \end{align}
     and
     \begin{align}\label{b1}
     n^{\frac{2}{3}}\left(\frac{\varrho_n(\hat{\B}_1)-E_{+}}{\gamma_0}
     \right)\Rightarrow TW_1,
     \end{align}
    where $\underline{E}_+$ and $\underline{\gamma}_0$ are defined in Corollary \ref{cor2.4}, $E_{+}$ and ${\gamma_0}$ are obtained by \eqref{E+} and \eqref{GAMMA0}, respectively, when $\T$ is replaced by $(\D_2^{(n)})^{-1}$.
\end{corollary}
\begin{remark}
In simulation, $\underline{E}_+$ and $\underline{\gamma}_0$ are reassigned to $((\sqrt{n-1}+\sqrt{p})/\sqrt{n})^2$ and $(\sqrt{n-1}+\sqrt{p})(1/\sqrt{n^2(n-1)}+1/\sqrt{n^2p})^{1/3}$ respectively, which are obtained 
by Theorem 1.1 in \cite{Johnstone}. This is due to: \begin{enumerate}[i]
  \item  Simulations behave better, especially for small $n,p$. \item The limiting distribution does not change. \end{enumerate}
\end{remark}
From this corollary, we know that for any data matrix $\B$ shifted by the mean matrix $\bar{\B}$, 
the normalized largest eigenvalue of its permuted matrices tends to follow the TW law.  Note that for a data matrix $\B$, the auxiliary matrix $\D_2^{(n)}$ is known, and thus the parameters $E_+$ and $\gamma_0$ can be directly obtained. 
For illustration, we generate the data matrix $\B=\bm \Psi$, where 
$\tilde{\T}={\bm I}_p$, $\bm \Phi$ is a matrix consisting of i.i.d. standard Gaussian random variables. The dimensional settings are $p = 400, 800, 1600, 3200, 4800$
and the sample size is $n = 3p/2$. 
Each $\B$ is permuted $5,000$ times. The empirical quantiles of $\varrho_n(\hat{\B}_2)$ of the permuted matrices are provided in Table \ref{tablenew}. From these results, we can find that although there is distortion for the quantiles when $p$ and $n$ are small, this distortion fades away as $p$ and $n$ increase. This indicates that in practice, if the data at hand are centralized, the critical values obtained by PA can be estimated by the quantiles of the TW distribution when $p$ and $n$ are large.  
This will eliminate the massive permutation computations, while also getting relatively accurate quantiles. We will discuss the asymptotic distribution of the largest eigenvalue of permuted sample covariance matrices
without centralization, as well as other parallel analysis methods in Section \ref{conclusion}.
\begin{table}
      \centering
      \caption{The percentiles of the normalized largest eigenvalues for permuted matrices}
      \begin{tabular}{c c c c c c c r}\toprule
      Percentile&TW law&$p=400$&$p=800$&$p=1600$&$p=3200$&$p=4800$\\\cmidrule(lr){1-1}\cmidrule(lr){2-2}\cmidrule(lr){3-3}\cmidrule(lr){4-4}\cmidrule(lr){5-5}\cmidrule(lr){6-6}\cmidrule(lr){7-7}
      $5\%$ &-3.1880&-3.4281&-3.3621&-3.3511&-3.3664&-3.3108\\ 
      $50 \%$ &-1.2680 & -1.6086&-1.5405&-1.4960&-1.4475&-1.3835\\
      $95\%$ &0.9765& 0.5968& 0.6401&0.7339& 0.8246& 0.8795\\
      \bottomrule
      \end{tabular}
      \label{tablenew}
    \end{table}

  \subsection{Real data 
example
}\label{realdata}
    We consider the HGDP dataset (e.g., 
    \cite{Li} and \cite{DO}). The purpose of collecting this dataset is to evaluate the diversity in the patterns of genetic variation across the globe.
    We use the Centre d'Etude du Polymorphisme Humain panel, in which SNP data were collected for 1043 samples representing 51 populations from Africa, Europe, Asia, Oceania and  America. The data can be obtained from http://www.hagsc.org/hgdp/data/hgdp.zip.

    We use the same processing pipeline as \cite{DO}. The setting can be found at github.com/dobriban/DPA.
    The dataset has $n=1043$ observations, and  $p=9730$ SNPs on chromosome 22, which are standardized. Thus we have a $p\times n$ data matrix $\B$.   
%
    Figure \ref{figure5} shows the numerical behavior of the limiting distribution of $\varrho_n(\hat{\B}_2)$ based on 5,000 independent permutations.
    We find that the numerical distribution fits the theoretical distribution (i.e., the TW distribution) comparatively well. 
    \begin{figure}[h]
      \centering
      \includegraphics[height=4cm,width=8cm]{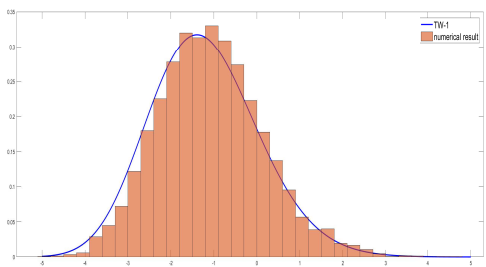}
      \caption{Numerical behavior of the permuted sample covariance matrix on the HGDP data.}
      \label{figure5}
    \end{figure}

\section{Technical proofs}\label{TP}

    \subsection{Proof strategy}
    The proof of Theorem \ref{comparison} is completed with the aid of a general strategy presented in \cite{KAY} for the covariance type matrices, which itself is an adaptation of the method originally described in \cite{ERD} for Wigner matrices. Generally, to prove the TW law for the largest eigenvalue, first, one needs to prove a local law for the spectral distribution, which controls the location of the eigenvalues on an optimal local scale. Second, with the aid of the local law, one needs to perform Green function comparison between the matrix of interest and a certain reference matrix ensemble, whose edge spectral behavior is already known. In \cite{ERD}, an extended criterion of the local law for the covariance type of matrices with independent columns (or rows) was given, see Theorem 3.6 in \cite{Pillai and Yin}. It allows one to relax the independence assumption on the entries within columns (or rows) to a certain extent, as long as some large deviation estimates
    hold for certain linear and quadratic forms of each column (or row) of the data matrix; see Lemma 3.4 in \cite{Pillai and Yin}. 

    Recall the sample covariance matrix ${\bS}_{n}$ defined in \eqref{al2}. Our main task is to show a large deviation estimate for each row of $\X_n$. Once the large deviation estimate is established, we can use the criterion in Proposition 5.1 in \cite{Wen} to conclude the weak local law of ${\bS}_{n}$ and Lemma 5.6 in \cite{KAY} for strong local law. In fact, Green function comparison can be completed similarly to the approach in \cite{ERD} by choosing an appropriate reference matrix ensemble. The reference matrix that is chosen in this paper is the traditional unbiased estimator of the population covariance matrix.

    \subsection{Notations}
    For two positive quantities $A_n$ and $B_n$ that depend on $n$, we use the notation $A_n\asymp B_n$ to mean that $C^{-1}A_n\leq B_n\leq CA_n$ for some positive constant $C>1$. We need the following probability comparison definition from \cite{ERD2}.

    Let $\mathcal{X}\equiv \mathcal{X}^{(N)}$ and $\mathcal{Y}\equiv \mathcal{Y}^{(N)}$ be two sequences of nonnegative random variables. We say that $\mathcal{Y}$ stochastically dominates $\mathcal{X}$ if, for all (small) $\epsilon>0$ and (large) $D>0$,
    \begin{eqnarray}\label{intro9}
    \mathbb{P}(\mathcal{X}^{(N)}>N^\epsilon\mathcal{Y}^{(N)})\leq N^{-D},
    \end{eqnarray}
    for sufficiently large $N>N_0(\epsilon,D)$, and we write $\mathcal{X}\prec\mathcal{Y}$ or $\mathcal{X}=O_\prec(\mathcal{Y})$. When $\mathcal{X}^{(N)}$ and $\mathcal{Y}^{(N)}$ depend on a parameter $v\in \mathcal{V}$ (typically an index label or a spectral parameter), then $\mathcal{X}\prec\mathcal{Y}$ uniformly in $v\in \mathcal{V}$, which means that the threshold $N_0(\epsilon,D)$ can be chosen independently of $v$.
    We use the symbols $O(\cdot)$ and $o(\cdot)$  for the standard big-O and little-o notations, respectively. We use $c$ and $C$ to denote strictly positive constants that do not depend on $N$. Their values may change from line to line. For any matrix $\A$, we denote $\|\A\|$ as its operator norm, while for any vector $\bm v$, we use $\|\bm v\|$ to denote its $L_2$-norm. In addition, we use $\ba_i$ to represent the $i$-th row of matrix $\A$.

  \subsection{Global law of $ S_n$}
  In this section, we prove Theorem \ref{th1MP} (Global law of $S^{(n)}$).
    \begin{proof}[Proof of Theorem \ref{th1MP}]
        When $\T$ is identity,
        \cite{Bai08} proved that $F_n$ is asymptotically given by the standard MP law. For general $\T$, we have
       \begin{align*}
       \mathbb E w_{ij}^2&=nt_i\mathbb E x_{ij}^2=t_i, ~~
       \mathbb E w_{i1}w_{i2}=nt_i\mathbb E x_{i1}x_{i2}=-\frac{t_i}{N-1},
       \end{align*}
       \begin{align*}
       \mathbb E w_{i1}^2w_{i2}w_{i3}&=\frac{1}{N(N-1)(N-2)}\sum_{1\leq s\neq t\neq l\leq N}u_{is}^2u_{it}u_{il}\\
       &=\frac{2}{(N-1)(N-2)}(\sum_{s=1}^Nu_{is}^4/N)-\frac{N}{(N-1)(N-2)}t_i^2,\\
       \mathbb E w_{i1}w_{i2}w_{i3}w_{i4}&=\frac{1}{N(N-1)(N-2)(N-3)}\sum_{1\leq s\neq t\neq l\neq m\leq N}u_{is}u_{it}u_{il}u_{im}\\
       &=\frac{1}{(N-1)(N-2)(N-3)}(-6(\sum_{s=1}^Nu_{is}^4/N)+3Nt_i^2),\\
       \mathbb E (w_{i1}^2-\mathbb Ew_{i1}^2)(w_{i2}^2-\mathbb Ew_{i2}^2)&=\frac{1}{N(N-1)}\sum_{1\leq s\neq t\leq N}u_{is}^2u_{it}^2-t_i^2\\
       &=\frac{1}{N-1}(t_i^2-\sum_{s=1}^Nu_{is}^4/N).
       \end{align*}
       Therefore, \eqref{th11} follows from Corollary 1.1 in \cite{Bai08}, which completes the proof of Theorem \ref{th1MP}.
    \end{proof}

    \subsection{Local law of $S_n$}
    In this section, our goal is to prove a strong
    local law for matrix ${\bS}_n$.  To achieve this goal, we first establish some large deviation estimates for certain linear and quadratic forms of $\x_i$,  which are rows of the matrix $\X_n$. 
    We now introduce a lemma on the stochastic domination of high-order moments.
    \begin{lemma}\label{le1}(Lemma 7.1 in \cite{Benaych})
    Suppose that the deterministic control parameter $\Phi$ satisfies $\Phi \geq n^{-C}$ for some constant $C>0$ and that for all $p$, there is a constant $C_p$ such that the nonnegative random variable $X$ satisfies $\mathbb E X^p \leq n^{C_p}$. Then, we have the following equivalence
    $$
    X\prec \Phi \Longleftrightarrow  \mathbb E X^n \prec \Phi^n
    $$
   for any fixed $n\in \mathbb N.$
    \end{lemma}
    
    For simplicity, we set $\frac{1}{\sqrt{n}}\bv_i=(v_1,\ldots,v_n)$ and $\x_i=(x_1,\ldots,x_n)$ as the $i$-th row of $\frac{1}{\sqrt{n}} \V_{p \times n}$ and $\X_n$, respectively. We emphasize that in this subsection, $N=n$.
    \subsubsection{Large deviation estimates}
    \begin{lemma}\label{large deviation}(Large deviation lemma).
    For any deterministic vector $\A={A_i}\in \mathbb{C}_n$ and deterministic  matrix $\B={B_{ij}}\in \mathbb{C}_{n\times n}$, we have
    \begin{eqnarray}\label{large deviation1}
    |\sum_{i=1}^nx_iA_i|\prec(\|\A\|^2/n)^{1/2},
    \end{eqnarray}
    \begin{eqnarray}\label{large deviation2}
    |\sum_{i=1}^nB_{ii}x_i^2-n^{-1}{\rm Tr}\B|\prec n^{-1}(\sum_{i=1}^n|B_{ii}|^2)^{1/2},
    \end{eqnarray}
    \begin{eqnarray}\label{large deviation3}
    |\sum_{i\neq j}^nx_i B_{ij}x_j+\frac{1}{n(n-1)}\sum_{i \neq j}B_{ij}|\prec n^{-1}(\sum_{i\neq j}^n|B_{ij}|^2)^{1/2}.
    \end{eqnarray}
    \end{lemma}

    \begin{proof}[Proof of Lemma \ref{large deviation}]
    The proof of Lemma \ref{large deviation} is very similar to the proof of Proposition 2.1 in \cite{Baospe}.
     
     By Chebyshev's inequality
    $$
    \mathbb P(|x_i|\geq n^{\epsilon}n^{-\frac{1}{2}})\leq \mathbb E \frac{|\sqrt{n}x_i|^q}{n^{q\epsilon}},
    $$
    and \eqref{mmmm}, we have $|x_i|=O_{\prec}(n^{-1/2})$.

    For \eqref{large deviation1}, we first construct a sequence of martingale differences. Define the filtration
    $$
    \mathcal{F}_0=\varnothing,~~~\mathcal{F}_l:=\sigma\{x_1,\ldots,x_l\},~~~l\in[1,n],
    $$
    and set
    \begin{eqnarray}\label{large deviation11}
    \mathcal{M}_l=\sum_{j=1}^nA_j[\mathbb E(x_j|\mathcal{F}_l)-\mathbb E(x_j|\mathcal{F}_{l-1})].
    \end{eqnarray}
    Conditioned on $\mathcal{F}_l$, $x_i$ is uniformly distributed for all $x_i\in \{x_{l+1},\ldots,x_n\}$ by sampling without replacement. Hence, we derive

    \begin{eqnarray}\label{large deviation12}
    \mathbb E(x_j|\mathcal{F}_l)= (\sum_{i=1}^n v_i-\sum_{i=1}^l x_i)/(n-l)=-\sum_{i=1}^l x_i/(n-l),~~~j\geq l+1,
    \end{eqnarray}
    where we used the fact that $\sum_{j=1}^nx_j=\sum_{j=1}^nv_j=0$.
    Furthermore, because $\mathbb E(x_j|\mathcal{F}_l)=x_j,~ j\in[1,l]$ and $\mathbb E(x_j|\mathcal{F}_{l-1})=x_j,~ j\in[1,l-1]$, $\mathcal{M}_l$ can be rewritten as follows:
    \begin{align}\label{large deviation13}
    \nonumber&\mathcal{M}_l
    =x_lA_l+\sum_{j=l+1}^nA_j \mathbb E(x_j|\mathcal{F}_l)-\sum_{j=l}^nA_j\mathbb E(x_j|\mathcal{F}_{l-1})\\
    \nonumber=&A_lx_l-x_l\frac{\sum_{j=l+1}^nA_j}{n-l}+A_l\frac{\sum_{k=1}^{l-1}x_k}{n-l+1}+\sum_{j=l+1}^nA_j\frac{-\sum_{k=1}^{l-1}x_k}{n-l}
    -\sum_{j=l+1}^nA_j\frac{-\sum_{k=1}^{l-1}x_k}{n-l+1}\\
    =&(A_l-\frac{1}{n-l}\sum_{j=l+1}^nA_j)(x_l+\frac{1}{n-l+1}\sum_{k=1}^{l-1}x_k).
    \end{align}

    Using the bound $|x_k|=O_{\prec}(n^{-1/2})$ and the fact that
    \begin{eqnarray}\label{buchong}
    |\sum_{k=1}^{l-1}x_k|\prec n^{-1/2} \min(l-1,n-l+1),
    \end{eqnarray}
    we deduce that
    \begin{eqnarray}\label{large deviation14}
    |\mathcal{M}_l|\prec n^{-1/2}(|A_l|+\frac{1}{n-l}\sum_{j=l+1}^n|A_j|).
    \end{eqnarray}
    By the Burkholder inequality, we derive that for any fixed positive integer $q>1$,
    \begin{eqnarray}\label{large deviation15}
      \mathbb E|\sum_{l=1}^n\mathcal{M}_l|^q\leq K_q \mathbb E(\sum_{l=1}^n\mathcal{M}_l^2)^{q/2},
    \end{eqnarray}
    where $K_q$ is finite depending on $q$. Now, we  estimate $\sum_{l=1}^n \mathcal{M}_l^2$. According to \eqref{large deviation14}, we have
    \begin{align}\label{large deviation16}
    \nonumber\sum_{l=1}^n\mathcal{M}_l^2=&\sum_{l=1}^{n-1}\mathcal{M}_l^2\prec n^{-1}\sum_{l=1}^{n-1}[\frac{1}{(n-l)^2}(\sum_{j=l+1}^n|A_j|)^2+|A_l|^2]\\
    \nonumber\prec& n^{-1}\sum_{l=1}^{n-1}[\frac{1}{(n-l)^2}(\sum_{j=l+1}^n|A_j|)^2+|A_l|^2]\prec n^{-1}\sum_{l=1}^{n-1}(\frac{1}{n-l}\sum_{j=l+1}^n|A_j|^2+|A_l|^2)\\
    \prec&n^{-1}\log n\|A\|^2\prec n^{-1}\|A\|^2.
    \end{align}
    Plugging \eqref{large deviation16} into \eqref{large deviation15}, using Lemma \ref{le1} and Markov's inequality, we can find \eqref{large deviation1}.

    For \eqref{large deviation2}, we construct a similar sequence of martingale differences as
    \begin{equation}\label{large deviation21}
    \mathcal{N}_l=\sum_{j=1}^nB_{jj}[\mathbb E(x_j^2|F_l)-\mathbb E(x_j^2|F_{l-1})].
    \end{equation}
    $\mathcal{N}_n=0$ since $\mathcal{F}_n=\mathcal{F}_{n-1}$. Conditioned on $\mathcal{F}_l$, $x_i$ is uniformly distributed for all $x_i\in \{x_{l+1},\ldots,x_n\}$. Hence,
    \begin{eqnarray}\label{large deviation22}
    \mathbb E(x_j^2|\mathcal{F}_l)= \frac{\sum_{i_=1}^n v_i^2-\sum_{i=1}^l x_i^2}{n-l}=\frac{1-\sum_{i=1}^l x_i^2}{n-l},~~~j\geq l+1.
    \end{eqnarray}
    Applying \eqref{large deviation22}, we deduce that
    \begin{align}\label{large deviation23}
    \nonumber\mathcal{N}_l=&B_{ll}x_l^2+\sum_{j=l+1}^nB_{jj}\mathbb E(x_j^2|\mathcal{F}_l)-\sum_{j=l}^nB_{jj}\mathbb E(x_j^2|\mathcal{F}_{l-1})\\
    =&(B_{ll}-\frac{1}{n-l}\sum_{j=l+1}^nB_{jj})(x_l^2+\frac{1}{n-l+1}(1-\sum_{k=1}^{l-1}x_k^2)).
    \end{align}
    Based on the fact that $|x_k|=O_{\prec}(n^{-1/2})$, we have
    $
    |\frac{1}{n-l}(1-\sum_{k=1}^{l-1}x_k^2)|\prec n^{-1}.
    $
    Thus,
    \begin{eqnarray}\label{large deviation24}
    |\mathcal{N}_l|\prec n^{-1}(|B_{ll}|+\frac{1}{n-l}\sum_{j=l+1}^n|B_{jj}|).
    \end{eqnarray}

    The remaining proof of \eqref{large deviation2} is the same as that for \eqref{large deviation1}. Thus, we omit the details.

    Finally, we prove \eqref{large deviation3}. Again, we first set
    \begin{eqnarray}\label{large deviation31}
    \mathcal{R}_l=\sum_{i<j}B_{ij}[\mathbb E(x_ix_j|\mathcal{F}_l)-\mathbb E(x_ix_j|\mathcal{F}_{l-1})].
    \end{eqnarray}
    Considering the different cases $i\leq l-1,~j=l$, $i\leq l-1,~j\geq l+1 $, $i=l,~ j\geq l+1$, $j>i\geq l+1$, and $i<j\leq l-1$, we rewrite $\mathcal{R}_l=\mathcal{R}_{l1}+\mathcal{R}_{l2}+\mathcal{R}_{l3}+\mathcal{R}_{l4}+\mathcal{R}_{l5}$, where
    \begin{align*}
    &\mathcal{R}_{l1}=\sum_{i=1}^{l-1}B_{il}x_i[x_l-\mathbb E(x_l|\mathcal{F}_{l-1})];~
    \mathcal{R}_{l2}=\sum_{i=1}^{l-1}\sum_{j=l+1}^{n}B_{ij}x_i[\mathbb E(x_j|\mathcal{F}_l)-\mathbb  E(x_j|\mathcal{F}_{l-1})];\\
    &\mathcal{R}_{l3}=\sum_{j=l+1}^{n}B_{lj}[\mathbb E(x_j|\mathcal{F}_l)x_l-\mathbb E(x_jx_l|\mathcal{F}_{l-1})];\\
    &\mathcal{R}_{l4}=\sum_{i=l+1}^{n}\sum_{j=i+1}^{n}B_{ij}[\mathbb E(x_ix_j|\mathcal{F}_l)-\mathbb E(x_ix_j|\mathcal{F}_{l-1})];\\
    &\mathcal{R}_{l5}=\sum_{i<j\leq l-1} B_{ij}[\mathbb E(x_ix_j|\mathcal{F}_l)-\mathbb E(x_ix_j|\mathcal{F}_{l-1})]=0.
    \end{align*}
    Applying \eqref{large deviation1} and setting $A=(\underbrace{1,\ldots,1}_l,\underbrace{0,\ldots,0}_{n-l})^{\top}$ and $A=(\underbrace{0,\ldots,0}_l,\underbrace{1,\ldots,1}_{n-l})^{\top}$, we have
    \begin{eqnarray}\label{large deviation32}
    |\sum_{j=1}^{l-1}x_j|=|\sum_{j=l}^{n}x_j|\prec n^{-1/2}\min\{\sqrt{l-1},\sqrt{n-l+1}\}.
    \end{eqnarray}
    Therefore, we have
    \begin{eqnarray}
      \mathbb E(x_j|\mathcal{F}_{l-1})\prec 1/\sqrt{n(n-l+1)},~~j\geq l.
    \end{eqnarray}
    Moreover, for $i,j\geq l$ and $i\neq j$, similar to \eqref{large deviation12} and \eqref{large deviation22}, we have
    \begin{align}\label{large deviation33}
    \nonumber&\mathbb E(x_ix_j|\mathcal{F}_{l-1})\\
    \nonumber=&\frac{1}{(n-l+1)(n-l)}\{\sum_{i\neq j}x_ix_j-\sum_{i,j\leq l-1,i\neq j}x_ix_j-2\sum_{i\leq l-1,j\geq l}x_ix_j\}\\
    \nonumber=&\frac{1}{(n-l+1)(n-l)}\{(\sum_ix_i)^2-\sum_ix_i^2-(\sum_{i\leq l-1}x_i)^2+\sum_{i\leq l-1}x_i^2\\
    \nonumber&-2\sum_{i\leq l-1}x_i(\sum_{j=1}^nx_j-\sum_{j\leq l-1}x_j)\}\\
    =&\frac{1}{(n-l+1)(n-l)}\{(\sum_{i\leq l-1}x_i)^2-(1-\sum_{i\leq l-1}x_i^2)\}.
    \end{align}
    In addition, we find that
    \begin{eqnarray}\label{large deviation34}
    |1-\sum_{i=1}^{l-1}x_i^2|\prec n^{-1} (n-l+1).
    \end{eqnarray}
    Thus, by \eqref{large deviation32}, \eqref{large deviation33} and \eqref{large deviation34}, we derive the following for $l\in[1,n]$.
    \begin{eqnarray}\label{large deviation35}
      \mathbb  E(x_ix_j|\mathcal{F}_{l-1})\prec n^{-1}(n-l)^{-1}.
    \end{eqnarray}
    By the Burkholder inequality, we obtain for any fixed positive integer $q>1$,
    \begin{eqnarray}\label{large deviation36}
      \mathbb E|\sum_{l=1}^n\mathcal{R}_l|^q\leq K_q \mathbb E(\sum_{l=1}^n\mathcal{R}_l^2)^{q/2},
    \end{eqnarray}
    and by the  Minkowski inequality, we have
    \begin{eqnarray}\label{large deviation366}
    (\mathbb E(\sum_{l=1}^n\mathcal{R}_l^2)^{q/2})^{2/q}\leq \sum_{l=1}^n(\mathbb E|\mathcal{R}_l|^q)^{2/q}.
    \end{eqnarray}
    Hence, bounding $\mathbb{E}|\mathcal{R}_{l\alpha}|^q$ for $\alpha=1,2,3,4$ is sufficient.

    For $\mathbb E|\mathcal{R}_{l1}|^q$,
    \begin{align}\label{large deviation37}
    \nonumber \mathbb E|\mathcal{R}_{l1}|^q=&\mathbb{E}(|\sum_{i=1}^{l-1}B_{il}x_i|^q|x_l-\mathbb{E}(x_l|\mathcal{F}_{l-1})|^q)\\
    \prec& n^{-q/2}\mathbb{E}|\sum_{i=1}^{l-1}B_{il}x_i|^q\prec n^{-q}(\sum_{i=1}^{l-1}|B_{il}|^2)^{q/2},
    \end{align}
    which is due to \eqref{large deviation1} and $x_i=O_{\prec}(n^{-1/2})$.

    For $\mathbb E|\mathcal{R}_{l2}|^q$, from \eqref{large deviation12} and \eqref{buchong}, when $l\leq n$, we have
    \begin{eqnarray}\label{large deviation38}
    |\mathbb E(x_j|\mathcal{F}_l)-\mathbb{E}(x_j|\mathcal{F}_{l-1})|=\frac{1}{n-l}|x_l+\frac{\sum_{j=1}^{l-1}x_j}{n-l+1}|\prec\frac{1}{(n-l)\sqrt{n}}.
    \end{eqnarray}
    Thus, we have
    \begin{align*}
    |\mathcal{R}_{l2}|\prec& \frac{1}{(n-l)\sqrt{n}}\sum_{j=l+1}^{n}|\sum_{i=1}^{l-1}B_{ij}x_i|
    \prec\frac{1}{(n-l)n}\sum_{j=l+1}^{n}(\sum_{i=1}^{l-1}|B_{ij}|^2)^{1/2}\\
    \prec&\frac{\sqrt{n-l}}{(n-l)n}(\sum_{j=l+1}^{n}\sum_{i=1}^{l-1}|B_{ij}|^2)^{1/2}\prec\frac{1}{n\sqrt{n-l}}(\sum_{j=l+1}^{n}\sum_{i=1}^{l-1}|B_{ij}|^2)^{1/2}.
    \end{align*}
    Hence, we have
    \begin{eqnarray}\label{large deviation39}
      \mathbb E|\mathcal{R}_{l2}|^q\prec n^{-q}(n-l)^{-q/2}(\sum_{j=l+1}^{n}\sum_{i=1}^{l-1}|B_{ij}|^2)^{q/2}.
    \end{eqnarray}

    Next, we  estimate $\mathbb{E}|\mathcal{R}_{l3}|^q$. Obviously, according to \eqref{large deviation32} and \eqref{large deviation35},
    \begin{eqnarray*}
    |\mathcal{R}_{l3}|\prec \sum_{j=l+1}^n|B_{lj}|/(n\sqrt{n-l})\prec (\sum_{j=l+1}^n |B_{lj}|^2)^{1/2}/n,
    \end{eqnarray*}
    In other words,
    \begin{eqnarray}\label{large deviation310}
      \mathbb E|\mathcal{R}_{l3}|^q \prec (\sum_{j=l+1}^n |B_{lj}|^2)^{q/2}/n^q.
    \end{eqnarray}

    Now, we estimate $\mathbb E|\mathcal{R}_{l4}|^q$. From \eqref{large deviation32}, \eqref{large deviation33} and \eqref{large deviation34}, we derive that for $i,j>l$,
    \begin{align*}
    &|\mathbb E(x_ix_j|\mathcal{F}_l)-\mathbb E(x_ix_j|\mathcal{F}_{l-1})|\\
    =&|\frac{1}{(n-l-1)(n-l)(n-l+1)}(\{(\sum_{i\leq l-1}x_i)^2-(1-\sum_{i\leq l-1}x_i^2)\})\\
    &+\frac{2}{(n-l)(n-l-1)}x_l\sum_{j=1}^{l}x_j|\prec\frac{1}{n(n-l)^{3/2}}.
    \end{align*}
    Hence, we obtain
    \begin{eqnarray}\label{large deviation311}
    \mathcal{R}_{l4}\prec \frac{1}{n(n-l)^{3/2}}\sum_{i=l+1}^{n}\sum_{j=i+1}^{n}|B_{ij}|\prec\frac{1}{n(n-l)^{1/2}}(\sum_{i=l+1}^{n}\sum_{j=i+1}^{n}|B_{ij}|^2)^{1/2}.
    \end{eqnarray}
    Therefore,
    \begin{eqnarray}\label{large deviation312}
      \mathbb E|\mathcal{R}_{l4}|^q\prec \frac{1}{n^q(n-l)^{q/2}}(\sum_{i=l+1}^{n}\sum_{j=i+1}^{n}|B_{ij}|^2)^{q/2}.
    \end{eqnarray}

    Finally, from \eqref{large deviation37}, \eqref{large deviation39}, \eqref{large deviation310} and \eqref{large deviation312}, we conclude that
    \begin{align*}
    \nonumber(\mathbb E|\mathcal{R}_l|^q)^{2/q}=&(\mathbb E(|\sum_{\alpha=1}^4\mathcal{R}_{l\alpha}|^2)^{q/2})^{2/q}\leq 4(\mathbb E(\sum_{\alpha=1}^4\mathcal{R}_{l\alpha}^2)^{q/2})^{2/q}\leq4\sum_{\alpha=1}^4(\mathbb E|\mathcal{R}_{l\alpha}^q|)^{2/q}\\
    \nonumber\prec& n^{-2}(\sum_{i=1}^{l-1}|B_{il}|^2+\sum_{j=l+1}^n |B_{lj}|^2\\
    \nonumber& +\frac{1}{n-l}\sum_{j=l+1}^{n}\sum_{i=1}^{l-1}|B_{ij}|^2+\frac{1}{n-l}\sum_{i=l+1}^{n}\sum_{j=i+1}^{n}|B_{ij}|^2)\\
    \prec& n^{-2}(\sum_{i=1}^{l-1}|B_{il}|^2+\sum_{j=l+1}^n |B_{lj}|^2+\frac{1}{n-l}\sum_{j=l+1}^{n}\sum_{i=1}^{j-1}|B_{ij}|^2).
    \end{align*}
    Since
    \begin{align*}
    &\sum_{l=1}^{n}\sum_{j=l+1}^{n}\sum_{i=1}^{j-1}\frac{1}{n-l}|B_{ij}|^2=\sum_{j=1}^{n}\sum_{i=1}^{j-1}(\sum_{l=1}^{j-1}\frac{1}{n-l})|B_{ij}|^2
    \leq C\log n(\sum_{j=1}^{n}\sum_{i=1}^{j-1})|B_{ij}|^2\prec\sum_{i<j}|B_{ij}|^2,
    \end{align*}
    we conclude that
    \begin{eqnarray*}
      \mathbb E|\sum_{l=1}^n\mathcal{R}_l|^q\prec n^{-q}(\sum_{i<j}|B_{ij}|^2)^{q/2}.
    \end{eqnarray*}
    Hence, by Markov's inequality, we have
    \begin{eqnarray}\label{large deviation4}
    |\sum_{i\neq j}x_i B_{ij}x_j-\mathbb E(\sum_{i\neq j}x_i B_{ij}x_j)|\prec n^{-1}(\sum_{i\neq j}|B_{ij}|^2)^{1/2}.
    \end{eqnarray}
    By the fact that $\mathbb{E}x_ix_j=-\frac{1}{n(n-1)},~i\neq j$,  
    Lemma \ref{large deviation} is concluded.
    \end{proof}

    Furthermore, for each $i=1,\ldots,p$, we let $\bm \xi_{i}=\{\xi_{i1},\ldots,\xi_{in}\}$ be a random vector with i.i.d. components that are uniformly distributed on $\bv_i$, i.e., $\mathbb P(\xi_{ik}=v_{ij})=n^{-1}$, $k=1,\ldots,n,~j=1,\ldots,n$. Note that $x_{ij}$'s are also identically distributed
    as $\xi_{ij}$'s, but $x_{ij}$'s are correlated. Set
    \begin{eqnarray}\label{large deviation5}
    \tilde{\x}_i=\bm \xi_i\bm \Sigma^{\frac{1}{2}},
    \end{eqnarray}
    where
    \begin{eqnarray}\label{large deviation6}
    \bm \Sigma=\frac{n}{n-1}\bm I_n-\frac{1}{n-1}\e_n\e_n^{\top}.
    \end{eqnarray}

    Let $\bm \Xi$ and $\tilde{\X}_n$ be the $p\times n$ matrices with $\bm {\xi}_i$ and $\tilde{\x}_i$ as their $i$-th rows, respectively. Introduce
    \begin{eqnarray}\label{large deviation7}
    \tilde{\bS}=\T^{\frac{1}{2}}\tilde{\X}_n\tilde{\X}_n^{\top}\T^{\frac{1}{2}}=\T^{\frac{1}{2}}\bm\Xi\bm\Sigma\bm\Xi^{\top}\T^{\frac{1}{2}}.
    \end{eqnarray}
    Similarly, set $\tilde{\y}_i=\sqrt{t_i}\tilde{\x}_i$, $\tilde{\Y}_n=(\tilde{\y}_1^{\top},\ldots,\tilde{\y}_p^{\top})^{\top}$ and  rewrite \eqref{large deviation7} as
    \begin{eqnarray}\label{large deviation8}
      \tilde{\bS}=\tilde{\Y}_n\tilde{\Y}_n^{\top}.
    \end{eqnarray}

    Observe that $\tilde{\bS}$ is the classical sample covariance matrix in statistics theory, which is centralized by sample mean and scaled by $n-1$, although in random matrix theory the simplified model $\T^{\frac{1}{2}}\bm\Xi\bm\Xi^{\top}\T^{\frac{1}{2}}$ is considered more often. Since $\T^{\frac{1}{2}}\bm\Xi\bm\Sigma\bm\Xi^{\top}\T^{\frac{1}{2}}$ is simply a rank one perturbation of $\frac{n}{n-1}\T^{\frac{1}{2}}\bm\Xi\bm\Xi^{\top}\T^{\frac{1}{2}}$, the empirical spectral distribution of $\tilde{\bS}$ almost surely also weakly converges to  a probability distribution, whose Stieltjes transform is given by \eqref{th11}.

    Observe that the entries of $\bm \Xi$ are i.i.d.. Then, using the large deviation estimates for linear and quadratic forms of the i.i.d. random variables (e.g., Corollary B.3 in \cite{ERD}), we see that \eqref{large deviation1},  \eqref{large deviation2} and \eqref{large deviation3} still hold if we replace $\x_i$ with $\tilde{\x}_i$.

    \subsubsection{Strong local law for $S_n$}
    
    We denote $\gamma_1\geq\cdots\geq \gamma_{p\wedge n}$ as the ordered $p$-quantiles of $F_{c_n,H_n}$, i.e., $\gamma_{j}$ is the smallest real number such that
    \begin{eqnarray}\label{local law1}
    \int_{-\infty}^{\gamma_{j}}{\rm d}F_{c_n,H_n}(x)=\frac{p-j+1}{p},~j=1,\ldots p\wedge n.
    \end{eqnarray}
   
    Recall matrix ${\bS}_{n}$ defined in \eqref{syx} and matrix $\tilde{\bS}$ defined in \eqref{large deviation8}. We introduce some intermediate matrices between ${\bS}_{n}$ and $\tilde{\bS}$. Starting from $\Y_n$, we replace $\y_i$'s with $\tilde{\y}_i$'s one by one and obtain a sequence of intermediate matrices
    \begin{eqnarray}\label{local law3}
    \Y_n(\Y_0),~~\Y_1,\ldots,\Y_l,\Y_{l+1},\ldots,\Y_{p-1},~~\tilde{\Y}_n(\Y_{p}).
    \end{eqnarray}
    Correspondingly, we set
    \begin{align}\label{local law4}
    \bS_l=\Y_l\Y_l^\top,~~&\bm G_l(z)=(\bS_l-z)^{-1},~~m_l(z)=\frac{1}{p}{\rm {\rm Tr}}\bm G_l(z)\\
    \mathcal{\bS}_l=\Y_l^{\top} \Y_l,~~&\mathcal{G}_l(z)=(\mathcal{\bS}_l-z)^{-1},~~\underline{m}_l(z)=\frac{1}{n}{\rm {\rm Tr}}\mathcal{G}_l(z),
    \end{align}
  
    When $l=0$, we have $\Y_0=\Y_n$, $m_0=m_n$; then, we simply write $\bS_0$; $\mathcal{\bS}_0$; $\bm G_0$; $\mathcal{G}_0$; $\underline{m}_0$ as $\bS$; $\mathcal{\bS}$ $\bm G$; $\mathcal{G}$; and $\underline{m}_n$.

    

     Furthermore, we introduce the following notation:
    
    \begin{eqnarray}\label{local law5}
    \Lambda_d=\max_{k,l}|\bm G_{kl}-\delta_{kl}\tilde{m}_k|,~~\Lambda_o=\max_{k\neq l}|\bm G_{kl}|,~~\Lambda=|m_n-\tilde{m}|=c_n^{-1}|\underline{m}_n-\tilde{\underline{m}}|,
    \end{eqnarray}
  
   where $\delta_{kl}$ denotes the Kronecker delta, i.e. $\delta_{kl}=1$
   if $k=l$, and $\delta_{kl}=0$ if $k\neq l$, and $\tilde{m}_k=\frac{1}{t_k(1-c_n-c_nz\tilde{m})-z}=-\frac{1}{z+z\tilde{\underline{m}}t_k}$. The last equality is derived by \eqref{local law_c2}.
    
    For a small positive $\bm c, \epsilon$ and sufficiently large $C_+$, $(C_+ > E_+)$, we define the domain
    \begin{eqnarray}\label{local law6}
    D(\bm c,\epsilon):=\{z=E+i\eta\in \mathbb C^{+}:E_+-\bm c\leq E\leq C_+,~n^{-1+\epsilon}\leq \eta\leq 1\}.
    \end{eqnarray}
    The following lemma on $\tilde{\underline{m}}(z)$ is elementary.
    \begin{lemma}\label{mn}(Theorem 3.1 in \cite{Baopanzhou}) Under Assumption 2.1, there is a $c>0$ such that for any $z=E+i\eta\in D(c,0)$
    \begin{eqnarray}
    &&|\tilde{\underline{m}}(z)|\asymp 1,\\
    &&{\rm {\rm Im}} \tilde{\underline{m}}(z)\asymp\left\{
                       \begin{array}{ll}
                         \sqrt{\kappa+\eta},&~~~~if~E\in [E_+-c,E_++\eta),\\
                         \frac{\eta}{\sqrt{\kappa+\eta}},&~~~~if~E\geq E_++\eta,
                       \end{array}
                     \right.\\
    && |1+t\tilde{\underline{m}}(z)|>0,~~~t\in[t_p,t_1],            
    \end{eqnarray}
    where $\kappa\equiv \kappa(E)=|E-E_+|$.
    \end{lemma}

    We further define the control parameter
   $
    \Omega\equiv\Omega(z)=\sqrt{\frac{{\rm {\rm Im}} \tilde{m}(z)}{n\eta}}+\frac{1}{n\eta}.$ We claim that the following local law holds.
    \begin{lemma}\label{lema2}
    Under {Assumption \ref{ass1}}, we have that for any
    sufficiently small $\epsilon>0$,

    \begin{enumerate}
      \item (Entrywise local law):$$ \Lambda_d(z)\prec\Omega(z)
          $$
          holds uniformly on $D(c,\epsilon)$.
      \item (Average local law): $$
          \Lambda(z)\prec\frac{1}{n\eta}
          $$
      holds uniformly on $D(c,\epsilon)$, where
      $c$ is the constant in Lemma \ref{mn}.
      \item (Rigidity on the right edge): For $i\in[1,\delta p]$ with any sufficiently small constant $\delta\in(0,1)$, we have
          $$
          |\lambda_i(\bS_{n})-\gamma_i|\prec n^{-\frac{2}{3}}i^{-\frac{1}{3}}.
          $$
    \end{enumerate}
    All the above hold if we replace ${\bS}_{n}$ with $\bS_l$ for all $l=1,\ldots,p$.
    \end{lemma}
    The proof of Lemma \ref{lema2} is deferred to Section A.1 of the Supplementary Material (\cite{Hu23}).

    \subsection{Edge universality for $S_n$}\label{univer}
    In this section, we prove the edge universality of ${\bS}_{n}$ by the Green function comparison.

    Recall the intermediate matrices defined in \eqref{local law3} and the notation introduced in \eqref{local law4}. We aim to show the following
    lemma.
    \begin{lemma}\label{comlema1}
    Fix any $\gamma\in[1,\ldots,p]$, and choose $\epsilon$ as any sufficiently small constant. Let $E,E_1,E_2\in R$ satisfy $E_1<E_2$ and
    \begin{eqnarray}\label{comlema11}
|E-E_+|,~|E_1-E_+|,~|E_2-E_+|\leq n^{-\frac{2}{3}+\epsilon},
    \end{eqnarray}
    and set $\eta_0=n^{-\frac{2}{3}-\epsilon}$. Define $F:\mathbb{R}\rightarrow \mathbb{R}$ as a smooth function satisfying
    \begin{eqnarray*}
    \max_{x\in \mathbb{R}}|F^{(l)}(x)|(|x|+1)^{-C}\leq C,~~~l=1,2,3,4
    \end{eqnarray*}
    for some positive constant $C$. Then, there exists a constant $\delta>0$ such that for sufficiently large $n$, we have
    \begin{eqnarray}\label{comlema12}
    |\mathbb EF(n\eta_0{\rm Im}m_{\gamma}(z))
    -\mathbb EF(n\eta_0{\rm Im}m_{\gamma+1}(z))|\prec n^{-1-\delta},~~~~~~~z=E+i\eta_0,
    \end{eqnarray}
    and
    \begin{eqnarray}\label{comlema13}
    |\mathbb EF(n\int_{E_1}^{E_2}{\rm Im}m_{\gamma}(x+i\eta_0){\rm d}x)
    -\mathbb EF(n\int_{E_1}^{E_2}{\rm Im}m_{\gamma+1}(x+i\eta_0){\rm d}x)|\prec n^{-1-\delta}.
    \end{eqnarray}
    \end{lemma}

   The proof of Lemma \ref{comlema1} can be found in Section A.2 of the Supplementary Material (\cite{Hu23}). Using Lemma \ref{comlema1}, we can now prove Theorem \ref{comparison}.
   \begin{proof}[Proof of Theorem \ref{comparison}]
     Since the proof of Theorem \ref{comparison} is similar to that of Theorem 3.3 in \cite{Wen}, we only outline the general strategy here. First, we attempt to convert the problem to a comparison of the Stieltjes transform. As a result,
    for any $\eta>0$, we define
	\[
	\vartheta_{\eta}(x)=\frac{\eta}{\pi(x^{2}+\eta^{2})}=\frac{1}{\pi}\operatorname{Im}\frac{1}{x-i\eta}.
	\]
	We notice that for any $a,b\in\mathbb{R}$ with $a\le b$, the convolution
	of $\mathbf{1}_{[a,b]}$ and $\vartheta_{\eta}$ applied to the eigenvalues
	$\lambda_{i}$, $i=1,\dots,p$ yields that
	\[
	\sum_{i=1}^{p}\mathbf{1}_{[a,b]}*\vartheta_{\eta}(\lambda_{i})=\frac{p}{\pi}\int_{a}^{b}\operatorname{Im} m_{n}(x+i\eta){\rm d}x.
	\]
	In terms of functional calculus notation, we have
	\begin{eqnarray*}
		\sum_{i=1}^{p}\mathbf{1}_{[a,b]}(\lambda_{i}) =  {\rm Tr}\mathbf{1}_{[a,b]}({\bS}_n),\quad
		\sum_{i=1}^{p}\mathbf{1}_{[a,b]}*\vartheta_{\eta}(\lambda_{i})  =  {\rm Tr}\mathbf{1}_{[a,b]}*\vartheta_{\eta}({\bS}_n).
	\end{eqnarray*}
	For $a,b\in\mathbb{R}\cup\{-\infty,\infty\}$, define ${\cal N}(a,b)=\#(a\leq \lambda_j \leq b)$ as the number of eigenvalues of ${\bS}_n$ in $[a,b]$. Then, we have the following lemma.

	\begin{lemma}\label{lem:smooth approximation of indicator}(Lemma 4.1 in \cite{Pillai and Yin})
	Let $\varepsilon>0$
	be an arbitrarily small number. Set $E_{\varepsilon}=E_{+}+n^{-2/3+\varepsilon}$,
	$\ell_{1}=n^{-2/3-3\varepsilon}$ and $\eta_{1}=n^{-2/3-9\varepsilon}$.
	Then, for any $E$ satisfying $|E-E_{+}|\le\frac{3}{2}n^{-2/3+\varepsilon}$, it holds with high probability that
	\[|{\rm Tr}\mathbf{1}_{[E,E_{\varepsilon}]}({\bS}_n)-{\rm Tr}\mathbf{1}_{[E,E_{\varepsilon}]}*\vartheta_{\eta}({\bS}_n)|\le C(n^{-2\varepsilon}+{\cal N}(E-\ell_{1},E+\ell_{1})).\]
	\end{lemma}
	Lemma \ref{lem:smooth approximation of indicator} shows that ${\rm Tr}\mathbf{1}_{[a,b]}({\bS}_n)$ can be well approximated by its smoothed version ${\rm Tr}\mathbf{1}_{[a,b]}*\vartheta_{\eta}({\bS}_n)$ for $a,b$ around edge $E_{+}$ so that the problem can be converted to a comparison of the Stieltjes transform.
	
    Then, from Lemma \ref{comlema1}, Lemma \ref{lem:smooth approximation of indicator}, the square root behavior of $\underline{\rho}_0$ (see Section \ref{notionnn}), and the rigidity on the right edge (see Lemma \ref{lema2}), one can see that
    \begin{align}\label{comp5}
    \nonumber&\mathbb P\left(n^{2/3}(\lambda_1(\tilde{\bS})-E_+)\leq s-n^{-\varepsilon}\right)-n^{-\delta}\leq\mathbb P\left(n^{2/3}(\lambda_1({\bS}_n)-E_+)\leq s\right)\\
    \leq&\mathbb P\left(n^{2/3}(\lambda_1(\tilde{\bS})-E_+)\leq s+n^{-\varepsilon}\right)+n^{-\delta},
    \end{align}
    where $\tilde{\bS}$ is defined in \eqref{large deviation7}. 
    It is known from the eigenvalue sticking result (Theorem 2.7 of \cite{Blo} 
    and Section 11 in \cite{KAY}) that the edge universality for  $\tilde{\bS}$ holds. Specifically,
    \begin{align}\label{comp6}
    \nonumber&\mathbb P\left(n^{2/3}(\lambda_1(\Q)-E_+)\leq s-n^{-\varepsilon}\right)-n^{-\delta}\leq\mathbb P\left(n^{2/3}(\lambda_1(\tilde{\bS})-E_+)\leq s\right)\\
    \leq&\mathbb P\left(n^{2/3}(\lambda_1(\Q)-E_+)\leq s+n^{-\varepsilon}\right)+n^{-\delta},
    \end{align}
    where $\Q$ is defined in Theorem \ref{comparison}. Combining \eqref{comp5} and \eqref{comp6}, we can conclude the proof of Theorem \ref{comparison}.
   \end{proof}

\section{Conclusion and discussion}\label{conclusion}
In this paper, we focus on the eigenvalues of high-dimensional sample covariance matrices generated without replacement from finite populations. The largest eigenvalues are proven to follow the TW law. Based on the permutation method introduced by \cite{DE}, we consider an analogous strategy to compare the top singular
values of $\B$ and a certain percentile of the top singular values of its permuted matrices after centralization, where the latter can be attributed to our theoretical results. 

Technically, for the PA, one might be more interested in the asymptotic distribution of $\lambda_1(\frac{1}{n}(\B)_{\pi}(\B)_{\pi}^{\top})$, that is the largest eigenvalue of permuted sample covariance matrices without centralization. We believe that it would still tend to the TW law as $n,~p\to \infty$ by adding some more restrictions on $\bS$. However, it requires more technical work, thus we decide to leave it for future work.  Besides, \cite{DO} proposed derandomized PA (DPA), deflated DPA (DDPA) and DDPA+ methods,  which 
are faster and more reproducibly than the PA method.  
Essentially, our theoretical results can explain why the DPA method is a kind of parallel analysis from another perspective. For DDPA and DDPA+ methods,  both counter the shadowing problem of DPA by removing some largest factors one by one. We think that the idea of removing the poorly estimated singular vectors can also improve the accuracy of the quantiles of the TW distribution in practice. We will also leave the rigorous discussion of this property for future work.






\section*{ Acknowledgments}
The authors would like to thank Zhigang Bao for allowing them to use the method in his unpublished work \cite{Baospe}. We also thank the editor, AE and referee for their careful reading and invaluable comments, which have greatly improved the quality of the paper. 


\section*{Funding}
Jiang Hu's research was supported by NSFC Nos.\ 12171078, 11971097, 12292980, 12292982. Shaochen Wang's research was supported by NSFC No.\ 11801181,
Guangdong Basic and Applied Basic Research Foundation No.\ 2018A030310358. 
Yangchun Zhang's research was supported by Shanghai Sailing Program 21YF1413500. 
Wang Zhou's research was partially supported by a grant A-0004803-00-00 at the National University of Singapore.

\end{document}